\documentclass[10pt,reqno]{amsart}

\usepackage[utf8]{inputenc}
\usepackage[english]{babel}
\usepackage{amsmath,amsfonts,amssymb,amsthm,shuffle}
\usepackage[T1]{fontenc}
\usepackage{lmodern}
\usepackage{mathtools}
\usepackage[titletoc]{appendix}

\usepackage[top=3.5cm,bottom=3.5cm,left=3.6cm,right=3.6cm]{geometry}

\usepackage[dvipsnames]{xcolor}
\usepackage[hyperindex=true,frenchlinks=true,colorlinks=true,
citecolor=Mahogany,linkcolor=DarkOrchid,urlcolor=Tan,linktocpage,
pagebackref=true]{hyperref}

\usepackage{tikz}
\usetikzlibrary{shapes}
\usetikzlibrary{fit}
\usetikzlibrary{decorations.pathmorphing}

\usepackage{dsfont}
\usepackage{wasysym}
\usepackage{stmaryrd}
\usepackage{cite}
\usepackage{subfigure}
\usepackage{multirow}
\usepackage{enumitem}
\usepackage{multicol}

\title[]{An exponential inequality for $U$-statistics of i.i.d. data}
\author{Davide Giraudo}
\date{\today}

\numberwithin{equation}{section}
\renewcommand{\leq}{\leqslant}
\renewcommand{\geq}{\geqslant}

\newtheorem{Theorem}{Theorem}[section]
\newtheorem{Th\'eor\`eme}{Th\'eor\`eme}[section]
\newtheorem{Proposition}[Theorem]{Proposition}
\newtheorem{Lemma}[Theorem]{Lemma}

\newtheorem{Definition}[Theorem]{Definition}
\newtheorem{D\'efinition}[Th\'eor\`eme]{D\'efinition}
\newtheorem{Corollary}[Theorem]{Corollary}

\theoremstyle{remark}
\newtheorem{Remark}[Theorem]{Remark}

\tikzstyle{Vertex}=[circle,draw=LimeGreen!80,fill=LimeGreen!8,
inner sep=1pt,minimum size=2mm,line width=1pt,font=\scriptsize]
\tikzstyle{Node}=[Vertex,draw=RoyalBlue!80,fill=RoyalBlue!8,inner sep=1.5pt]
\tikzstyle{Leaf}=[rectangle,draw=Black!70,fill=Black!16,
inner sep=0pt,minimum size=1mm,line width=1.25pt]
\tikzstyle{Edge}=[Maroon!80,cap=round,line width=1pt]
\tikzstyle{Mark1}=[draw=BrickRed!80,fill=BrickRed!8]
\tikzstyle{Mark2}=[draw=BurntOrange!80,fill=BurntOrange!8]
\tikzstyle{EdgeRew}=[->,RedOrange!80,cap=round,thick]

\newcommand{\Fca}{\mathcal{F}}
\newcommand{\Gca}{\mathcal{G}}
\newcommand{\Hca}{\mathcal{H}}

\newcommand{\Rca}{\mathcal{R}}
\newcommand{\Sca}{\mathcal{S}}

\newcommand \ens[1]{\left\{ #1\right\}}
\newcommand \R{\mathbb R}

\newcommand \PP{\mathbb P}

\newcommand{\E}[1]{\mathbb E\left[#1\right]}

\newcommand \Z{\mathbb Z}

\newcommand \abs[1]{\left|#1\right|}
\newcommand \eps{\varepsilon}

\newcommand{\pr}[1]{\left(#1\right)}
\newcommand{\norm}[1]{\left\lVert #1 \right\rVert}

\newcommand{\conv}{\leq_{\mathrm{conv}}}
\newcommand{\som}[2]{\sum_{#1}^{#2}}
\begin{document}


\begin{abstract}
 We establish an exponential inequality for degenerated $U$-statistics of order $r$ of i.i.d. data. This 
 inequality gives a control of the tail of the maxima  obsulute values 
 of the $U$-statistic by the sum of two terms: an exponential term and one involving the tail 
 of $h\pr{X_1,\dots,X_r}$. We also give a version for not necessarily degenerated $U$-statistics 
 having a symmetric kernel and furnish an application to the convergence rates 
 in the Marcinkiewicz law of large numbers. Application to invariance principle in Hölder 
 spaces is also considered.
\end{abstract}
\maketitle

 \section{Exponential inequalities for $U$-statistics}
 
 \subsection{Goal of the paper}
  Let $\pr{\Omega,\Fca,\PP}$ be a probability space,
  $\pr{S,\Sca}$ be a measurable space and let $r\geq 1$ be an integer. 
 Let also $h\colon S^r\to \R$ be a measurable function and $\pr{X_i}_{i\geq 1}$
 an  i.i.d. sequence where $X_i\colon \Omega\to S$. 
 The $U$-statistic (of order $r$)
 of kernel $h$ and data $\pr{X_i}_{i\geq 1}$ is defined as 
 \begin{equation}\label{eq:definition_U_stat}
 U_{r,n}\pr{h}:=\sum_{i\in I_n^r}h\pr{X_{i_1},\dots, X_{i_r}},
 \end{equation}
 where 
 \begin{equation}\label{eq:definition_de_I_n_r}
  I_n^r:=\ens{i\in\Z^r, 1\leq i_1<i_2<\dots<i_r\leq n   }.
 \end{equation}
When there is no confusion with the kernel or the order, we 
will simply write $U_n$.
 
Our goal is to control the tail of $U_n$. More precisely, we would 
like to give a control on the following quantity 
\begin{equation}
\PP\ens{\max_{r\leq n\leq N}\abs{U_n}>N^{r/2}x}, \quad N\geq r, x>0,
\end{equation}
independently on $N$, and with the help of a functional of the tail of 
$h\pr{X_1,\dots,X_r}$. 
This type of inequalities has been studied in the case where the 
kernel $h$ is bounded in
 \cite{MR0144363,MR1235426,MR1323145,MR1130366,MR2791056,MR2336595}. 
An extension to the case of $U$-statisitics of the form $\sum_{i\in I_n^r}h_i\pr{X_{i_1},
\dots,X_{i_r}}$ has been considered in \cite{MR1857312,MR2073426}.
We will then provide applications to the convergence rates in the law of large numbers. Moreover, 
the established inequality is also a good tool to verify tightness criterion for partial sum 
processes in Hölder spaces.

 \subsection{Statement for degenerated $U$-statistics}

We will first assume that this $U$-statistic is degenerated, in the sense that 
for all $s_1,\dots,s_{r-1}\in S$ and for all $q\in \ens{1,\dots,r}$,  
\begin{equation}\label{eq:condition_de_degeneration}
\E{h\pr{v_q\pr{s_1,\dots,s_{r-1},X_0}}}=0,
\end{equation} 
where $v_q\pr{s_1,\dots,s_{r-1},X_0}$ is a vector of elements of 
$S^r$ whose $q-1$ first components are respectively $s_1,\dots,s_{q-1}$, 
the component $q$ is $X_0$ and the remaining ones are 
$s_{q+1},\dots,s_{r-1}$. Notice that here, we are not making any symmetry 
assumption on the kernel $h$.

\begin{Theorem}\label{thm:inegalite_deviation_U_stats}
Let  $r\geq 1$ be an integer, $\pr{S,\Sca}$ be measurable space, $h\colon 
S^r\to \R$ be a measurable function (with $S^r$ induced with the product 
$\sigma$-algebra) and let $\pr{X_i}_{i\geq 1}$ be an i.i.d. sequence of 
$S$-valued random variables. Suppose that \eqref{eq:condition_de_degeneration} 
holds for all $s_1,\dots,s_{r-1}\in S$ and for all $q\in \ens{1,\dots,r}$. Then 
the following inequality holds for all positive $x$ and all $y$ such that 
$x/y>3^r$:
\begin{multline}\label{eq:inegalite_exp_cas_degenere}
\PP\ens{\max_{r\leq n\leq N}\abs{U_n}>N^{r/2}x} 
\leq A_{r}\exp\pr{-  \frac 12\pr{\frac{x}y}^{ \frac{2}{r    }}}
\\+B_{r}\int_{1 }^{+\infty}
\PP\ens{\abs{h\pr{X_1,\dots,X_r}}> y u ^{1/2} C_{r}   } 
\pr{1+\ln\pr{ u}}^{q_{r}}\mathrm du,
\end{multline}
 where 

\begin{equation}
A_{r}:=4\pr{1-2^{-r}}
\end{equation}
\begin{equation}
B_{r}:=2\prod_{i=1}^{r}\frac{2^{i-1 }}{\kappa_{i-1   }}
\pr{1+\ln\frac{\kappa_{2 \pr{i-1}    }}{\kappa_{  i-1}  }}^{i-1}
\end{equation}
where $\kappa_q$ is defined for a non-negative $q$ by 
\begin{equation}
 \kappa_q=\begin{cases}
           1&\mbox{ if }q\leq 1\\
           e^{q-1}/q^q&\mbox{ if }q> 1;
          \end{cases}
\end{equation}
\begin{equation}
C_{r}:=4^{-r/2}\prod_{i=1}^r\kappa_{ i-1  }^{1/2}
\end{equation}
and 
\begin{equation}
q_{r}:=\frac{\pr{r-1}r}{2}.
\end{equation}
\end{Theorem}
 
Let us give some comments on the result. 
 
\begin{Remark}
The second term of the
 right-hand-side of \eqref{eq:inegalite_exp_cas_degenere} is of order 
 $\E{Y^2 \log\pr{1+Y}^{q_{r}}\mathbf 1\ens{Y>1}}$, where $Y= 
 \abs{h\pr{X_1,\dots,X_r}}/y$. This implies that Theorem~\ref{thm:inegalite_deviation_U_stats} 
 is useful only if  $\E{Y^2 \log\pr{1+Y}^{q_{r}} }$ is finite. 
\end{Remark}
\begin{Remark}
Proposition~2.3 in \cite{MR1235426} gives also an exponential inequality 
when $h$ is bounded. Theorem~\ref{thm:inegalite_deviation_U_stats} gives 
a similar result up to the involved constants. 
\end{Remark}
\begin{Remark}
The bigger $y$ is, the smaller the second term of the right hand side of \eqref{eq:inegalite_exp_cas_degenere} is but the higher the first term of the right hand side is. Therefore, 
when applying this inequality, we try to use it with an appropriate value of $y$.
\end{Remark}
\begin{Remark}
The right hand side of \eqref{eq:inegalite_exp_cas_degenere} is independent of $N$. Therefore, 
when Theorem~\ref{thm:inegalite_deviation_U_stats} is applied with a stronger normalization 
than $N^{r/2}$, we can get a decay in $N$. 
\end{Remark}

\begin{Remark}
However, this inequality does not seem 
to be sufficient to deduce directly a satisfactory result for the 
 law of the iterated logarithms. Indeed, one would need to control 
$\sum_{N\geq r}\PP\ens{\max_{r\leq n\leq 2^N}\abs{U_n}>2^{Nr/2} a \sqrt{\log N} }  $, 
which forces the choice $x=  a \sqrt{\log N} $ for each fixed $N$ and one should choose $y<3^rx$ hence 
we would need exponential moments for $\abs{h\pr{X_1,\dots,X_r}}$, which is of course 
suboptimal.
\end{Remark}

%

\subsection{Statement for general $U$-statistics with symmetric kernel}

We now would like to give a result similar to 
Theorem~\ref{thm:inegalite_deviation_U_stats} but 
without assuming that the $U$-statistic is degenerated. To this aim, we use the 
Hoeffding decomposition. 
Here we assume that the kernel $h\colon S^r\to\R$ is symmetric in the 
sense that for all $x_1,\dots,x_r\in S$ and all permutation 
$\sigma\colon \ens{1,\dots,r}\to \ens{1,\dots,r}$, 
\begin{equation}
h\pr{x_{\sigma\pr{1}},\dots,x_{\sigma\pr{r}}}= 
h\pr{x_1,\dots,x_r}.
\end{equation}

The $U$-statistic involved in Theorem~\ref{thm:inegalite_deviation_U_stats} are such that 
\begin{equation}
\E{h\pr{X_1,\dots,X_r}\mid \sigma\pr{X_i, 1\leq i\leq r,i\neq j   }}=0
\end{equation}
almost surely for all $j\in \ens{1,\dots,r}$. In order to extend the result to a larger class 
of $U$-statistic, we need to define a more general notion of degeneratedness. 

\begin{Definition}
Let $\pr{X_k}_{k\geq 1}$ be an i.i.d. sequence with values in the 
measurable space $\pr{S,\Sca}$. Let $h\colon S^r\to \R$ be a measurable 
function such that $Y:= h\pr{X_1,\dots,X_r}$ is integrable. Denote 
$\Fca_j$ the $\sigma$-algebra generated by the random variables 
$X_1,\dots, X_j$. 
We say that the $U$-statistic $U_{r,n}\pr{h}$ is degenerated of order 
$i-1$ for some $i\in\ens{2,\dots,r}$ if $\E{Y\mid \Fca_{i-1}}=0$ 
almost surely but $\E{Y\mid \Fca_{i}}$ is not equal to zero almost surely.
\end{Definition}

We express the $U$-statistic associated to this kernel $h$ as a sum of 
$U$-statistics of order $k$ with symmetric kernel. Define 
\begin{equation}
\pi_{k,r}h\pr{x_1,\dots,x_k}:=
\pr{\delta_{x_1}-\PP_{X_1}} \dots \pr{\delta_{x_k}-\PP_{X_1}}\PP_{X_1}^{r-k}h, 
\end{equation}
where $Q_1\dots Q_r h$ is defined as 
\begin{equation}
Q_1\dots Q_r h= \int\dots\int h\pr{x_1,\dots,x_r}dQ_1\pr{x_1}\dots 
dQ_r\pr{x_r}.
\end{equation}
Then the following equality holds:
\begin{equation}\label{eq:decomposition_de_Hoeffding}
\binom{n}{r}U_{r,n}\pr{h-\theta}=\sum_{k=1}^k\binom{r}{k}\binom{n}{k}U_{k,n}\pr{h_k},
\end{equation}
where $\theta:=\E{h\pr{X_1,\dots,X_n}}$ 
and $U_{k,n}\pr{h_k}$ is a generated $U$-statistic of order $k$.

If $U_{r,n}\pr{h}$ is degenerated of order $i-1$ for 
some $i\in\ens{2,\dots,r}$, then 
the first $i-1$ terms in \eqref{eq:decomposition_de_Hoeffding} vanish (and $\theta=0$) 
hence 
\begin{equation}\label{eq:decomposition_de_Hoeffding_deg_ordre_i_1}
\binom{n}{r}U_{r,n}\pr{h }=\sum_{k=i}^r\binom{r}{k}\binom{n}{k}U_{k,n}\pr{h_k}.
\end{equation}
Therefore, appling Theorem~\ref{thm:inegalite_deviation_U_stats} 
to each degenerated $U$-statistic $U_{k,n}\pr{h_k}$ gives 
the following result.

\begin{Corollary}\label{cor:extension_inegalite_de_deviation}
Let  $r\geq 1$ be an integer,   $\pr{S,\Sca}$ be measurable space, $h\colon 
S^r\to \R$ be a measurable function (with $S^r$ induced with the product 
$\sigma$-algebra) and let $\pr{X_i}_{i\geq 1}$ be an i.i.d. sequence of 
$S$-valued random variables. Assume that $h$ is degenerated of order 
$i-1$ for some $i\in\ens{2,\dots,r}$. Then for all positive $x$ 
and $y$ such that $x/y>3^r$, 
\begin{multline}\label{eq:extension_inegalite_de_deviation_non_deg}
\PP\ens{\max_{r\leq n\leq N}\abs{U_n}>N^{r-\frac{i}{2}}x} 
\leq A_{r}\exp\pr{- \frac 12\pr{\frac{x}y}^{ \frac{2}{i   }}}
\\+B_{r}\sum_{k=i}^r\int_{1 }^{+\infty}
\PP\ens{\abs{\E{h\pr{X_1,\dots,X_r}\mid X_1,\dots,X_k   }}> y 
N^{ \frac{ k-i }{2}}u ^{1/2} C_{r}   } 
\pr{1+\ln\pr{ u}}^{q_{k}}\mathrm du, 
\end{multline}
where $q_{k }=\frac{ \pr{k-1}k}{2}$ and the constants 
$A_{r}$, $B_{r}$ and $C_{r}$ depend only on $r$.
\end{Corollary}

When $h\pr{X_1,\dots,X_n}$ has a finite exponential moments, 
it turns out that a simpler upper bound can be given. 

\begin{Corollary}\label{cor:inegalite_moments_exp}
 Let  $r\geq 1$ be an integer and $\gamma>0$. There are constants 
 $x_{r,\gamma}$, $A_{r_\gamma}$ and $B_{r,\gamma}$ such that if 
 $\pr{S,\Sca}$ is a measurable space, $h\colon 
S^r\to \R$ is a measurable function (with $S^r$ induced with the product 
$\sigma$-algebra),   $\pr{X_i}_{i\geq 1}$ an i.i.d. sequence of 
$S$-valued random variables and $h$ is degenerated of order 
$i-1$ for some $i\in\ens{2,\dots,r}$, then for all $x\geq x_{r,i,\gamma}
:=3^{r\frac{2+i\gamma}{i\gamma}}2^{-1/\gamma} C_{r }^{-1} $ (with $C_{r }$ 
like in Corollary~\ref{cor:extension_inegalite_de_deviation})
 and 
all $N\geq r$, 

\begin{equation}\label{eq:inegalite_Ustats_moments_expo}
\PP\ens{\max_{r\leq n\leq N}\abs{U_n}>N^{r -i/2}x} 
\leq A_{r,\gamma}
 \exp\pr{-   B_{r,\gamma}     x ^{\frac{2\gamma}{i\gamma+2}   }      }
\E{\exp\pr{ \abs{h\pr{X_1,\dots,X_r} }^\gamma   }}.
\end{equation}

\end{Corollary}
 \begin{Remark}
Corollary~1.2 in \cite{MR1655931} gives a result in the same spirit, but with the 
following differences.
\begin{itemize}
\item Our inequality gives a bound on the tail probability of $\max_{r\leq n\leq N}\abs{U_n}$, 
while \cite{MR1655931} gives a control of the tail of $\abs{U_N}$.
\item Our inequality show explicitely how the 
right hand side depends on $\E{\exp\pr{ \abs{h\pr{X_1,\dots,X_r} }^\gamma   }}$. It 
allows in particular to apply the inequality to $R\cdot h$, $R>0$, instead of $h$ when 
$\E{\exp\pr{ R\abs{h\pr{X_1,\dots,X_r} }^\gamma   }}$ is finite.
\item The case $0<\gamma\leq 2$ was addressed in \cite{MR1655931}, whereas 
we cover the case $\gamma>0$.
\end{itemize} 
\end{Remark}

\subsection{Application to convergence rates in the strong law of large numbers}


Let $U_n$ be the $U$-statistic defined by 
\eqref{eq:definition_U_stat}. Suppose that 
$\E{\abs{h\pr{X_1,\dots,X_r}}}$ is finite. Then 

\begin{equation}
 \frac 1{n^r} \sum_{i\in I_n^r}
 \pr{h\pr{X_{i_1},\dots,X_{i_r}}-
 \E{h\pr{X_{i_1},\dots,X_{i_r}}}}\to 0\mbox{ a.s.}
\end{equation}

If we assume that the $U$-statistic is degenerated of 
order $i-1$ and if we impose more restrictive
integrability conditions, an other normalization than 
$n^r$ can be chosen.

\begin{Theorem}[Theorem 1 in \cite{MR1607435}]
 Let $h\colon \R^r\to\R$ be a symmetric function, $r\geq 2$ and 
 let $\pr{X_i}_{i\geq 1}$ be an i.i.d. sequence of random variables.
 Suppose that $h$ is degenerated of order $i-1$ where 
 $i\in\ens{2,\dots,r}$. Let $q\in\pr{1,2r/\pr{2r-i}}$. Suppose 
 that for all $j\in\ens{i,\dots,r}$, 
 \begin{equation}
  \E{h\pr{X_1,\dots,X_r}\mid X_1,\dots ,X_j}\in
  \mathbb L^{\frac{jq}{r-\pr{r-j}q}}.
 \end{equation}
Then $n^{-r/q}U_n\to 0$ almost surely.
\end{Theorem}

Information on the convergence rates in the Marcinkiewicz 
law of large number have been obtained in \cite{MR1227625} 
and results in the spirit of Baum-Katz \cite{MR198524} for partial sums
   has ben obtained in  \cite{MR614652,MR903815}, in all the cases under polynomial 
   moment conditions, that is, finiteness of $\E{\abs{h\pr{X_1,\dots,X_r}}^q}$ 
   for some $q$.

Our setting would also allow to derive result under similar assumptions but it seems 
that the inequality we obtain is not the most suitable in this context. 
However, under finite exponential moments, one can use 
 Corollary~\ref{cor:inegalite_moments_exp} in order to quantify 
the convergence.

\begin{Theorem}\label{thm_Baum_Katz_Ustats}
Let $\pr{S,\Sca}$ be a measurable space,  $h\colon S^r\to\R$ be a symmetric function, $r\geq 2$ and 
 let $\pr{X_i}_{i\geq 1}$ be an i.i.d. sequence of random variables with values in $S$.
 Suppose that $h$ is degenerated of order $i-1$ where 
 $i\in\ens{2,\dots,r}$. Let $\alpha\in \pr{r-i/2,r}$. If there exists a positive 
 $\gamma>0$ such that $\E{\exp\pr{R \abs{h\pr{X_1,\dots,X_r}}^\gamma  }}<+\infty$ is 
 finite for all $R$,  then
 \begin{equation}
 \forall \eps>0, \quad \sum_{N\geq 1}
 \exp\pr{
 2^{N\pr{\alpha-\pr{r-i/2}      }\frac{2\gamma}{i\gamma+2}    } 
 }\PP\ens{\max_{r\leq n\leq 2^N}\abs{U_n}>\eps 2^{N\alpha}}<+\infty.
 \end{equation}
 \end{Theorem}

An other way to measure the speed of convergence in the law of large numbers is by bounding 
the probabilities of large deviation. In the case of partial sums of an i.i.d. centered 
 sequence $\pr{X_i}_{i\geq 1}$, the involved probability is $\PP\ens{\max_{1\leq n\leq N}\abs{
 \sum_{i=1}^nX_i}>Nx  }$. Extension to martingale with bounded moments have been 
 investigated in \cite{MR1856684} and \cite{MR3005732}.

  In the context of the $U$-statistics, the normalization will depend on the 
 degree of degeneracy. 
 
 \begin{Theorem}\label{thm:large_deviation_Ustats}
Let $\pr{S,\Sca}$ be a measurable space,  $h\colon S^r\to\R$ be a symmetric function, $r\geq 2$ and 
 let $\pr{X_i}_{i\geq 1}$ be an i.i.d. sequence of random variables with values in $S$.
 Suppose that $h$ is degenerated of order $i-1$ where 
 $i\in\ens{2,\dots,r}$.  If there exists a positive 
 $\gamma>0$ such that $M:=\sup_{t>0}\exp\pr{t^\gamma} 
 \PP\ens{   \abs{h\pr{X_1,\dots,X_r}}>t }$ is finite,  then for all $x>0$, 
 \begin{equation}
 \PP\ens{\max_{r\leq n\leq N}\abs{U_n}> N^{r}x}\leq K_1
 \exp\pr{ -K_2 N^{\frac{i\gamma}{2+i\gamma}}x^{\frac{2\gamma }{2+i\gamma} }   },
 \end{equation}
 where $K_1$ and $K_2$ depend on $\gamma$ and $M$. 
 \end{Theorem}

\subsection{Weak invariance principle in Hölder spaces}
 \label{subsec:WIP_Holder}
 In this subsection, we will study an other limit theorem for $U$-statistics: 
 the functional central limit theorem in Hölder spaces. Given a $U$-statistic 
 $U_n$ of order $r$ and kernel $h\colon S^r\to \R$, we define a partial sum process
 by 
 \begin{equation}\label{eq:definition_processus_sommes_partielles}
 \sigma_n\pr{t}:= \frac 1{n^{r/2}}
 \pr{U_{[nt]}-\pr{nt-[nt]} \pr{U_{[nt]+1}- U_{[nt]}    }}, t\in [0,1], n\geq r,
 \end{equation}
 where for $x\in \R$, $[x]$ is the unique integer satisfying $[x]\leq x<[x]+1$. In other words, $\sigma_n\pr{k/n}=U_k$ 
 and the random function $t\mapsto \sigma_n\pr{t}$ is 
 affine on the intervals $\left[k/n,\pr{k+1}/n\right]$.
 In \cite{MR740907}, the convergence in distribution in the 
 Skorohod space $D[0,1]$ of the process $\pr{n^{-r/2}
 U_{[n\cdot ]}}_{n\geq r}$ is studied. In Corollary~1, it is shown that if $U_n$ is degenerated 
 of order $i-1$, $i\in\ens{2,\dots,r}$, then $\pr{n^{-i/2} U_{[n\cdot ]}}_{n\geq r}$ 
 converges in distribution to a process $I_i\pr{h_i}$ symbolically defined as 
 \begin{multline}
 I_i\pr{h_i}\pr{t}=\int\dots\int h_i\pr{x_1,\dots,x_i}\mathbf 1_{[0,t]}\pr{u_1}
 \dots \mathbf 1_{[0,t]}\pr{u_i}W\pr{dx_1,du_1}\dots W\pr{dx_i,du_i},
 \end{multline}
where $W$ denotes the Gaussian measure (see the Appendix~A.1 and A.2 of the 
paper \cite{MR740907}). For $i=2$, the limiting process admits the 
expression $\sum_{j=1}^{+\infty} \lambda_j\pr{B_j^2\pr{t}-t}$, where 
$\pr{B_j\pr{\cdot}}_{j\geq 1}$ are independent standard Brownian motions and 
$\sum_{j=1}^{+\infty} \lambda_j^2$ is finite. In particular, such a process has 
path in Hölder spaces and would be also the limiting process for $\pr{\sigma_n}_{n\geq 1}$ 
when $r=2$.  Therefore, the study of the limiting behavior of 
$\pr{\sigma_n\pr{t}}_{n\geq r}$ in Hölder spaces can be considered. 

This question has been considered in the context of partial sum processes built 
on strictly stationary sequences of random variables, that is, of the form 
\begin{equation}\label{eq:def_processus_sommes_partielles_stationnaire}
W_n\pr{t}:= \frac 1{a_n}\pr{\sum_{i=1}^{[nt]}X_i+\pr{nt-[nt]}X_{[nt]+1}  },
\end{equation}
where $\pr{X_j}_{j\geq 1}$ is a strictly stationary centered sequence and $a_n\to +\infty$. The asymptotic
behaviour of such partial processes in $D[0,1]$ 
under dependence has concentrated a lot of effort; 
see for instence \cite{MR2206313} for a survey on the main results. 

Define $\mathcal R_i$ as the class of the real-valued functions $\rho$ defined on $[0,1]$
which can be expressed as $\rho\pr{t}=t^\alpha L\pr{1/t} $, where $L\colon [1,+\infty)\to \R$ 
is normalized slowly varying at infinity, positive and continuous, $\rho$ is increasing on 
$[0,1]$ and 
\begin{equation}
\lim_{\delta\to 0}\rho\pr{\delta}\delta^{-1/2}\pr{\ln\pr{\delta^{-1}} }^{-i/2}=+\infty.
\end{equation}
Notice that this implies that 
 $0<\alpha\leq 1/2$ and for $\alpha=1/2$, the constraint reads 
 \begin{equation}
 \lim_{\delta\to 0}L\pr{1/\delta} \pr{\ln\pr{\delta^{-1}} }^{-i/2}=+\infty.
 \end{equation} 
For example, if $c$ is such that the function $t\mapsto t^{1/2}\pr{\ln\pr{c/t}}^{\beta}$ is increasing, then 
the latter constraint forces $\beta>i/2$.

For $\rho\in\mathcal R_i$, we denote by $\Hca_\rho$ the Hölder space associated to the 
modulus of regularity $\rho$, that is, the set of function $x\colon [0,1]\to\R$ such that 
$\norm{x}_\rho:=\sup_{0\leq s<t\leq 1}\abs{x\pr{t}-x\pr{s}}/\rho\pr{t-s}+\abs{x\pr{0}}$ 
is finite. Instead of dealing with the convergence in $\Hca_\rho$, we will work with a subspace
which is more adapted to the study of convergence in distribution. Let 
\begin{equation}
\Hca_\rho^o:=\ens{x\colon [0,1]\to \R\mid \lim_{\delta\to 0}
\sup_{\substack{s,t\in [0,1]\\ 0<t-s<\delta  } }   
\frac{\abs{x\pr{t}-x\pr{s}}}{\rho\pr{t-s}}=0
}.
\end{equation}

The convergence of partial sum processes of the form \eqref{eq:def_processus_sommes_partielles_stationnaire} when $\pr{X_i}_{i\geq 1}$ is i.i.d. has 
been studied in \cite{MR2000642,MR2054586}. The convergence of $\pr{W_n\pr{\cdot}}_{n\geq 1}$
in $\Hca_\rho^o$ for $\rho\in\Rca_1$ holds if and only if 
\begin{equation}
\forall A>0, \lim_{t\to +\infty}t\PP\ens{\abs{X_1}
>At^{1/2}\rho\pr{1/t}
}=0.
\end{equation}
Generally, a strategy to prove such results is to establish the convergence on the finite dimensional 
distribution and prove tightness, which is usually the most difficult part. In Equation (1.3) in 
\cite{MR3615086}, a tightness criterion for partial sum processes of the form \eqref{eq:def_processus_sommes_partielles_stationnaire} with $\pr{X_j}_{j\geq 1}$ 
and $\rho$ of the form $t\mapsto t^\alpha$, $0<\alpha<1/2$. Its verification is done 
by using deviation inequalities, see for example \cite{MR3426520} or Section~3.3 in 
\cite{MR3583992}.

For the purpose of the study of the convergence of $\pr{\sigma_n\pr{\cdot}}_{n\geq 1}$ (defined 
by \eqref{eq:definition_processus_sommes_partielles}), we need to extend 
this criterion in two directions: to partial sum processes like in 
\eqref{eq:def_processus_sommes_partielles_stationnaire}
for which the sequence   $\pr{X_j}_{j\geq 1}$
is not necessarily stationnary and to the class of modulus of regularity $\Rca_1$.

\begin{Proposition}\label{prop:critere_de_tension}
Let $\pr{X_j}_{j\geq 1}$ be a sequence of random variables. Let $W_n$ be the partial sum process 
built on $ \pr{X_j}_{j\geq 1}$ defined by \eqref{eq:def_processus_sommes_partielles_stationnaire}. 
Let $\rho\in \Rca_i$ for some $i\geq 1$. Suppose that for all positive $\eps$, the following convergences 
hold:
\begin{equation}\label{eq:tightness_criterion}
\lim_{J\to +\infty}\limsup_{n\to +\infty}\sum_{j=J}^{[\log_2n]}\sum_{k=0}^{2^j-1}
   \PP\ens{\abs{S_{ [n\pr{k+1}2^{-j}    ]   }-S_{ [nk2^{-j}]   }  }   >a_n \eps\rho\pr{2^{-j}}   }=0 ;
\end{equation}
where 
$S_N:=\sum_{i=1}^NX_i$ and $\pr{ a_n}_{n\geq 1}$ is an increasing sequence diverging to infinity and
 such that 
$\sup_{n\geq 1}a_{2n}/a_n$ is finite. Then the partial sum process defined by \eqref{eq:def_processus_sommes_partielles_stationnaire} is tight in $\Hca_\rho^o$.
\end{Proposition}

Since the previous tightness criterion involves the tails of differences of partial sums, we have to 
establish a corresponding deviation inequality.

\begin{Proposition}\label{prop:deviation_accroissements}
Let  $r\geq 1$ be an integer, $\pr{S,\Sca}$ be measurable space, $h\colon 
S^r\to \R$ be a measurable function (with $S^r$ induced with the product 
$\sigma$-algebra) and let $\pr{X_i}_{i\geq 1}$ be an i.i.d. sequence of 
$S$-valued random variables. Suppose that \eqref{eq:condition_de_degeneration} 
holds for all $s_1,\dots,s_{r-1}\in S$ and for all $q\in \ens{1,\dots,r}$. Then 
the following inequality holds for all positive $x$ and all $y$ such that 
$x/y>3^r$ and all $n_2>n_1\geq r$:
\begin{multline}\label{eq:inegalite_exp_cas_degenere_differences}
\PP\ens{\frac 1{\sqrt{n_2-n_1}n_2^{\frac{r-1}2}}\abs{U_{n_2}-U_{n_1}}>x}
\leq A_{r}\exp\pr{-  \frac 12\pr{\frac{x}y}^{ \frac{2}{r    }}}
\\+B_{r}\int_{1 }^{+\infty}
\PP\ens{\abs{h\pr{X_1,\dots,X_r}}> y u ^{1/2} C_{r}   } 
\pr{1+\ln\pr{ u}}^{q_{r}}\mathrm du,
\end{multline}
where $A_r$, $B_r$ and $C_r$ depend only on $r$ and 
$q_r=r\pr{r-1}/2$.
\end{Proposition}

By combining this tightness criterion with the obtained deviation inequalities, we get the 
following function central limit theorem. 

\begin{Theorem}\label{thm:PI_Holderien}
Let  $r\geq 1$ be an integer,  $\pr{S,\Sca}$ be measurable space, $h\colon 
S^r\to \R$ be a symmetric measurable function (with $S^r$ induced with the product 
$\sigma$-algebra) and let $\pr{X_i}_{i\geq 1}$ be an i.i.d. sequence of 
$S$-valued random variables. Assume that $h$ is degenerated of order 
$i-1$ for some $i\in\ens{2,\dots,r}$ and that $h$ is symmetric. Let $\rho\in\Rca_r$. Assume that 
\begin{equation}\label{eq:condition_suffisante_WIP}
\forall c>0, \sum_{j=1}^{+\infty}
\int_1^{+\infty}\PP\ens{\abs{h\pr{X_1,\dots,X_r}}> c2^{j/2}\rho\pr{2^{-j}}j^{-r/2} u^{1/2}}
 \pr{1+\log u}^{q_r}du<+\infty.
\end{equation}

 Then 
\begin{equation}\label{eq:WIP_Ustats}
\frac 1{n^{i/2}}\pr{U_{[nt]}-\pr{nt-[nt]} \pr{U_{[nt]+1} -U_{[nt]}  }  }\to 
I_i\pr{h_i}\pr{t}
\end{equation}
in distribution in $\Hca_\rho^o\pr{[0,1]}$.

In particular, denoting $Y:= \abs{h\pr{X_1,\dots,X_r}}$, when $\rho\pr{t}=t^\alpha$, $0<\alpha<1/2$,  the condition 
\begin{equation}\label{eq:cond_suff_WIP_talpha}
\E{ Y^{\frac{1}{1/2-\alpha}} \pr{\log Y}^{r/2}   }<+\infty
\end{equation}
guarantees \eqref{eq:WIP_Ustats}.
When $\rho\pr{t}=t^{1/2}\pr{\log \pr{C/t}}^{\beta}$, $\beta>r/2$, the condition 
\begin{equation}\label{eq:cond_suff_WIP_t12}
\forall A>0, \E{\exp\pr{A  Y^{\frac 1{\beta-r/2}}  } }<+\infty
\end{equation}
guarantees \eqref{eq:WIP_Ustats}.

\end{Theorem}

\begin{Remark}
When $r=1$ and $\rho\pr{t}=t^\alpha$, we do not recover exactly the 
necessary and sufficient condition established in \cite{MR2000642}, which reads $
\lim_{t\to +\infty} t^{1/\pr{1/2-\alpha}}\PP\ens{Y>t}=0$. However, when 
$\rho\pr{t}=t^{1/2}\pr{\log \pr{C/t}}^{\beta}$, we recover the same condition.
\end{Remark}

 \section{Proofs}

 \subsection{Proof of Theorem~\ref{thm:inegalite_deviation_U_stats}}
Let us first give the general idea of proof of Theorem~\ref{thm:inegalite_deviation_U_stats}.
\begin{enumerate}
\item We will proceed by an induction argument on $r$. We will use
martingale inequality, which helps to control the tail maximum of 
partial sums of a martigale differences sequences with the help of the 
tail of the sum of squares and the sum of conditional variances. 
\item Using the notion of convex ordering that will be made explicit later, 
 the sum of squares and the sum of conditional variances satisfy 
a "convexity domination" hence we will be reduced to tails of $U$-statistics 
of the previous order. 
\item In order to perform the induction step, we will also need a rearragement 
of integrals of the type $\int_1^{+\infty}\int_1^{+\infty}
\PP\ens{Y>u/\pr{1+\ln u}^av}du\pr{1+\ln v}^b  dv$, whose treatment will be addressed later.
\end{enumerate} 
 
\subsubsection{Martingale inequality} 
 
We will formulate the martingale inequality we need to handle the 
tail of $U$-statistics. First, the following inequality is 
established  in \cite{MR3311214}.

\begin{Theorem}[Theorem~2.1 in \cite{MR3311214}]\label{thm:Fan_Grama_Liu}
Let $\pr{D_j}_{j\geq 1}$ be a martingale differences sequence 
with respect to the filtration $\pr{\Fca_i}_{i\geq 0}$. Suppose that 
there exists random variables $V_{i-1}$, $1\leq i\leq n$, which are non-negative, 
$\Fca_{i-1}$-measurable, non-negative functions $f$ and $g$ such that for some 
positive $\lambda$ and for all $i\in\ens{1,\dots,n}$,
\begin{equation}\label{eq:condition_sur_les_accroissements}
\E{\exp\pr{\lambda D_i-g\pr{\lambda}D_i^2}\mid \Fca_{i-1}}\leq 1+f\pr{\lambda}
V_{i-1}.
\end{equation}
Then for all $x$, $v$, $w>0$, 
\begin{multline}\label{eq:inegalite_Fan_grama_Liu}
\PP\pr{\bigcup_{k=1}^n \ens{\sum_{i=1}^kD_i\geq x}\cap 
\ens{\sum_{i=1}^kD_i^2\leq v^2   }\cap\ens{\sum_{i=1}^kV_{i-1}\leq w   }}\\
\leq \exp\pr{-\lambda x+g\pr{\lambda}v^2+f\pr{\lambda}w}.
\end{multline}
\end{Theorem}
 
 Notice that the condition \eqref{eq:condition_sur_les_accroissements} is satisfied 
 for all $\lambda>0$ when $f\pr{\lambda}=g\pr{\lambda}=\lambda^2/2$ and 
 $V_{i-1}=\E{D_i^2\mid\Fca_{i-1}}$. After having optimized over 
 $\lambda$ the right hand side of \eqref{eq:inegalite_Fan_grama_Liu} 
 and applying Theorem~\ref{thm:Fan_Grama_Liu} to $D_j$ and then to 
  $-D_j$, we get that 
 \begin{equation}
 \PP\ens{\max_{1\leq k\leq n}\abs{\sum_{i=1}^kD_i}>x   }
 \leq 2\exp\pr{-\frac{x^2}{2v^2}}+\PP\ens{\sum_{i=1}^n
 \pr{D_i^2+\E{D_i^2\mid \Fca_{i-1}}}>v^2}
 \end{equation}
  for all $x$ and $v>0$.

\subsubsection{Convex ordering}

Let us compare the tails of $Y$ and 
$  \E{Y\mid\Gca}$. By Markov's inequality, 
\begin{equation*}
 \PP\ens{\E{Y\mid\Gca}>x}
\leq \frac 1x\E{ \E{Y\mid\Gca}\mathbf{1}\ens{\E{Y\mid\Gca}>x}  }
= \frac 1x\E{ Y\mathbf{1}\ens{\E{Y\mid\Gca}>x}  }
\end{equation*}
 and splitting the last expectation over the sets where $Y\leq x/2$ and 
 $Y>x/2$ gives 
 \begin{equation}
 \PP\ens{\E{Y\mid\Gca}>x}\leq\frac 12 \PP\ens{\E{Y\mid\Gca}>x}+
 \frac 1x\E{ Y\mathbf{1}\ens{Y>x/2}  }
 \end{equation}
 hence 
  \begin{equation}
x \PP\ens{\E{Y\mid\Gca}>x}\leq\ 
 \E{ 2Y\mathbf{1}\ens{2Y>x}  }
 \end{equation}
and writing the last expectation as an integral over $\pr{0,+\infty}$ 
involving the tail of $Y$, we get that 
   \begin{equation}
x \PP\ens{\E{Y\mid\Gca}>x}\leq\ 
 x\PP\ens{2Y>x}+\int_x^{+\infty}\PP\ens{2Y>u}\mathrm{d}u.
 \end{equation}
 Bounding further $ x\PP\ens{2Y>x}$ by 
 $2\int_{x/2}^x\PP\ens{2Y>u}\mathrm{d}u$ finally gives 
 after the substitution $v=u/\pr{2x}$ that 
 \begin{equation}
 \PP\ens{\E{Y\mid\Gca}>x}\leq\ 
 4\int_{1/4}^{+\infty}\PP\ens{ Y>xv}\mathrm{d}v.
 \end{equation}
 hence one can control the tails of $\E{Y\mid\Gca}$ by 
a functional of those of $Y$. 

Therefore, if $Y$ and $Z$ are two non-negative random variables 
such that 
\begin{equation}\label{eq:ordre_convexe_esperance}
Z\leq \E{Y\mid Z}\mbox{ a.s.}
\end{equation}
then for all $x>0$, 
 \begin{equation}\label{eq:inegalites_queues_generale}
 \PP\ens{Z>x}\leq
  \int_{1}^{+\infty}\PP\ens{ Y>xv/4}\mathrm{d}v.
 \end{equation}
 
Observe also that if $Z$ and $Y$ are two non-negative random variables 
such that \eqref{eq:ordre_convexe_esperance} holds, then for each convex 
non-decreasing function $\varphi\colon \R_+\to\R_+$, 
\begin{equation}\label{eq:ordre_convexe_definition}
\E{\varphi\pr{Z}}\leq \E{\varphi\pr{Y}}.
\end{equation}
When \eqref{eq:ordre_convexe_definition} holds for each convex non-decreasing 
function, we write $Z\leq_{\mathrm{conv}} Y$.

By Theorem~6 in \cite{MR606989}, part (c),  if two random variables 
$Y$ and $Z$ defined on a common probability space 
$\pr{\Omega,\Fca,\PP}$
satisfy \eqref{eq:ordre_convexe_definition}, then there exists 
a probability space $\pr{\Omega',\Fca',\PP'}$ and random variables 
$Y'$ and $Z'$ such that
\begin{itemize}
\item for all real number $t$, $\PP'\ens{Y'\leq t}=\PP\ens{Y\leq t}$ 
and $\PP'\ens{Z'\leq t}=\PP\ens{Z\leq t}$;
\item the inequality $Z'\leq \mathbb{E}_{\PP'}\left[Y'\mid Z'\right]$ holds almost surely. 
\end{itemize}

Combining this with \eqref{eq:inegalites_queues_generale} gives the 
following lemma that will be used in the sequel.

\begin{Lemma}\label{lem:ordre_convexe}
Let $Y$ and $Z$ be two non-negative random variables such that for all convex 
non-decreasing function $\varphi\colon \R_+\to\R_+$, 
$\E{\varphi\pr{Z}}\leq \E{\varphi\pr{Y}}$. Then for all positive $x$, 
 \begin{equation}\label{eq:inegalites_queues_ordre_convexe}
 \PP\ens{Z>x}\leq
 \int_{1}^{+\infty}\PP\ens{ Y>xv/4}\mathrm{d}v.
 \end{equation}
\end{Lemma}

Using the previous lemma, one can control the tails of the maximum of a 
a martingale whose increments have common majorant 
for the order $\leq_{\mathrm{conv}} $.

\begin{Proposition}\label{prop:inegalite_deviation_martingales_dominee}
 Let $\pr{D_j}_{j\geq 1}$ be a martingale differences sequence 
with respect to the filtration $\pr{\Fca_i}_{i\geq 0}$.  Suppose that $\E{ D_i ^2}$ is finite for all 
$i\geq 1$. Suppose that there exists a random variable $Y$ such that 
for all $1\leq i\leq n$, $ D_i^2\leq_{\mathrm{conv}} Y^2$. 
 Then for all $x,y>0$ and each $n\geq 1$, the following 
inequality holds:
 \begin{equation}\label{eq:inegalite_max_martingales_moments_ordre_p_domination_sto}
 \PP\ens{\max_{1\leq k\leq n}\abs{\sum_{i=1}^kD_i}>x n^{1/2}  }
 \leq 2\exp\pr{-\frac 12\pr{\frac{x}{y}}^{2}  }\\ 
 + 2\int_{1 }^{+\infty}\PP\ens{ Y^2>y^2u/4}\mathrm{d}u.
 \end{equation}
\end{Proposition}
 
\begin{proof}
We apply Theorem~\ref{thm:Fan_Grama_Liu} in the following setting: 
$x$ is replaced by $xn^{1/2}$ and $v=n^{1/2}y$. We obtain that 
 \begin{multline} 
 \PP\ens{\max_{1\leq k\leq n}\abs{\sum_{i=1}^kD_i}>xn^{1/2}   }
 \leq 2\exp\pr{-\frac 12\pr{\frac xy}^{2   }  }  \\
    +\PP\ens{\sum_{i=1}^n
  D_i^2>ny^2}+\PP\ens{\sum_{i=1}^n\E{ D_i ^2\mid \Fca_{i-1}}>ny^2}.
 \end{multline}
 In order to control the last two terms, we will show that 
 \begin{equation}\label{eq:ordre_convexe_pour_variances_cond}
 \frac 1n \sum_{i=1}^nD_i^2\conv Y^2\mbox{ and }
 \frac 1n \sum_{i=1}^n\E{D_i^2\mid \Fca_{i-1}}\conv Y^2.
 \end{equation}
 Let $\varphi\colon\R_+\to \R_+$ be a convex non-decreasing function. Then 
 by convexity, 
 \begin{equation}
 \E{\varphi\pr{ \frac 1n \sum_{i=1}^nD_i^2}}\leq 
  \E{ \frac 1n \sum_{i=1}^n\varphi\pr{D_i^2}}
 \end{equation}
 and from the fact that $D_i^2\conv Y^2$, it follows  that $\E{  \varphi\pr{D_i^2}}\leq 
 \E{\varphi\pr{D_i^2}}$ hence 
  \begin{equation}
 \E{\varphi\pr{ \frac 1n \sum_{i=1}^nD_i^2}}\leq 
  \E{\varphi\pr{Y^2}}.
 \end{equation}
 Moreover, by convexity of $\varphi$ and Jensen's inequality, 
 \begin{equation*}
  \E{\varphi\pr{ \frac 1n \sum_{i=1}^n\E{D_i^2\mid\Fca_{i-1}   }}}\leq 
  \E{ \frac 1n \sum_{i=1}^n\varphi\pr{\E{D_i^2\mid\Fca_{i-1}   }}}
  \leq 
  \E{ \frac 1n \sum_{i=1}^n\E{\varphi\pr{D_i^2}\mid\Fca_{i-1}   }}
 \end{equation*}
 hence 
 \begin{equation}
 \E{\varphi\pr{ \frac 1n \sum_{i=1}^n\E{D_i^2\mid\Fca_{i-1}   }}}\leq 
  \frac 1n \sum_{i=1}^n\E{\varphi\pr{D_i^2}    }.
 \end{equation}
 By the same argument as before, 
  \begin{equation*}
  \E{\varphi\pr{ \frac 1n \sum_{i=1}^n\E{D_i^2\mid\Fca_{i-1}   }}}\leq 
 \E{\varphi\pr{Y^2}}
 \end{equation*}
which proves \eqref{eq:ordre_convexe_pour_variances_cond}. 
We conclude by applying Lemma~\ref{lem:ordre_convexe} 
twice, first to  $Z= n^{-1}\sum_{i=1}^n D_i^2$ and 
then to $Z= n^{-1}\sum_{i=1}^n\E{ D_i^2\mid\Fca_{i-1}}$. 
\end{proof} 
 
 \subsubsection{Step $r=1$}
 
As said before, the proof of Theorem~\ref{thm:inegalite_deviation_U_stats} 
is done by induction on $r$. In this subsubsection, we treat the case $r=1$.
In this case, $U_n=\sum_{i=1}^nh\pr{X_i}$ and condition
 \eqref{eq:condition_de_degeneration} means that  $h\pr{X_i}$ is centered. 
 Therefore, an application of Proposition~\ref{prop:inegalite_deviation_martingales_dominee} 
 to $D_j:= h\pr{X_j}$ gives the case $r=1$ of Theorem~\ref{thm:inegalite_deviation_U_stats}.
 
  \subsubsection{The induction step}
 
Let $A\pr{r}$ be the following assertion: "for each measurable space $\pr{S,\Sca}$, each measurable function 
 $h\colon 
S^r\to \R$ (with $S^r$ induced with the product 
$\sigma$-algebra) and each  i.i.d. sequence  $\pr{X_i}_{i\geq 1}$ of 
$S$-valued random variables such that  
\begin{equation} 
\E{h\pr{v_q\pr{x_1,\dots,x_{r-1},X_0}}}=0,
\end{equation} 
holds for all $s_1,\dots,s_{r-1}\in S$ and for all $q\in \ens{1,\dots,r}$,  
the following inequality holds for all positive $x$ and all $y$ such that 
$x/y>3^r$:
\begin{multline}
\PP\ens{\max_{r\leq n\leq N}\abs{U_n}>N^{r/2}x} 
\leq A_{r}\exp\pr{-\frac 12\pr{\frac{x}y}^{ \frac{2}{r    }}}
\\+B_{r}\int_{1}^{+\infty}
\PP\ens{\abs{h\pr{X_1,\dots,X_r}}> y v^{1/2}C_{r}     } 
\pr{1+\ln\pr{ v}}^{q_{r}}\mathrm dv".
\end{multline}
We have shown in the previous subsubsection that $A\pr{1}$ is true. 

Now, we assume that $A\pr{r-1}$ is true for some $r\geq 2$ 
and we will prove $A\pr{r}$. Let  $\pr{S,\Sca}$ be a measurable space, 
 $h\colon 
S^r\to \R$ (with $S^r$ induced with the product 
$\sigma$-algebra) be a measurable function,  $\pr{X_i}_{i\geq 1}$ be an i.i.d. sequence of 
$S$-valued random variables satisfying 
\begin{equation} 
\E{h\pr{v_q\pr{x_1,\dots,x_{r-1},X_0}}}=0,
\end{equation} 
 for all $s_1,\dots,s_{r-1}\in S$ and for all $q\in \ens{1,\dots,r}$,  
and  $x$ and $y$ be positive numbers such that  
$x/y>3^r$. We will control 
$\PP\ens{\max_{r\leq n\leq N}\abs{U_n}>N^{r/2}x}$ 
by a quantity that can be treated by the case of $U$-statistics 
of order $r-1$.

We define 
\begin{equation}
D_j:= \frac 1{N^{\frac{r-1}2}}\som{\pr{i_1,\dots,i_{r-1}}\in I_j^{r-1}}{} h\pr{X_{i_1},\dots,X_{i_{r-1}},X_j}, \quad j\geq r,
\end{equation}
 where the notation $I_{j}^{r-1}$ refers to \eqref{eq:definition_de_I_n_r} and 
 $D_j=0$ for $1\leq j\leq r-1$. Then for all $n\geq r$, $\frac 1{N^{r/2}}
 U_n=   \frac 1{N^{1/2}}\sum_{j=1}^nD_j$. Moreover, defining 
 $\Fca_j$ as the $\sigma$-algebra generated by the random variables $X_i$, $1\leq i\leq j$ for 
 $j\geq 1$ and $\Fca_0:=\ens{\emptyset, \Omega}$, 
 the sequence $\pr{D_j}_{j\geq 1}$ is a martingale differences 
 sequence with respect to the filtration $\pr{\Fca_j}_{j\geq 0}$. 
  Let 
  \begin{equation}
Y:= \frac 1{N^{\frac{r-1}2}}\max_{r-1\leq k\leq N}\abs{ U'_k},   
 \end{equation}
where 
\begin{equation}
U'_k:= \sum_{\pr{i_1,\dots,i_{r-1}}\in I_k^{r-1}}
h\pr{X_{i_1},\dots,X_{i_{r-1}},X}
\end{equation} 
and $X$ is a random variable which is independent of $\pr{X_i}_{i\geq 1}$ 
and has the same distribution as $X_0$. We will show that for all $j\geq r$, 
$D_i^2\conv Y^2$. First observe that since 
$X_j$ is independent of the vector $\pr{X_i}_{i=1}^{j-1}$, it follows that 
$D_j$ has the same distribution as 
\begin{equation}
D'_j:= \frac 1{N^{\frac{r-1}2}}\som{\pr{i_1,\dots,i_{r-1}}\in I_j^{r-1}}{}
 h\pr{X_{i_1},\dots,X_{i_{r-1}},X}, \quad j\geq r.
\end{equation} 

 Let $\varphi\colon\R_+\to \R_+$ be a convex non-decreasing function. Then 
 \begin{equation}\label{eq:lien_Dj_D'j}
 \E{\varphi\pr{D_i^2}}=\E{\varphi\pr{ {D'_j}^2}}.
 \end{equation}
Since 
\begin{equation}
 {D'_j} ^2= \frac 1{N^{r-1}} {U'_j}^2\leq Y^2
\end{equation}
we get from non-decreasingness of $\varphi$ that $\E{\varphi\pr{ {D'_j}^2}}
\leq \E{\varphi\pr{Y^2}}$ hence in view of \eqref{eq:lien_Dj_D'j}, it follows that 
$D_j^2\conv Y^2$. Writing 
\begin{equation}
\PP\ens{\max_{r\leq n\leq N}\abs{U_n}>N^{r/2}x}
=\PP\ens{\max_{1\leq n\leq N}\abs{\sum_{j=1}^n  D_j}>N^{1/2}x},
\end{equation}
we are in position to apply  Proposition~\ref{prop:inegalite_deviation_martingales_dominee} 
with $\widetilde{x}:=x$ and $\widetilde{y}:= x^{1-1/r}y^{1/r}$. 
We obtain 
\begin{equation}\label{eq:se_ramener_au_cas_r-1}
\PP\ens{\max_{r\leq n\leq N}\abs{U_n}>N^{r/2}x}
\leq 2\exp\pr{-\frac 12\pr{\frac{x}{y}}^{2}  } 
 + 2\int_{1 }^{+\infty}\PP\ens{ Y>x^{1-1/r}y^{1/r}\pr{u/4}^{1/2}}\mathrm{d}u.
\end{equation}
 In order to control the last term, we use the fact that 
 $X$ is independent of the sequence $\pr{X_i}_{i\geq 1}$ hence 
\begin{multline}\label{eq:preparation_de_l_etape_de_recurrence}
\PP\ens{ Y>x^{1-1/r}y^{1/r}\pr{u/4}^{1/2}}\\
=\int_{S}   \PP\ens{ \frac 1{N^{\frac{r-1}2}      }     
\max_{r-1\leq k\leq n} \abs{ \sum_{\pr{i_1,\dots,i_{r-1}}\in I_k^{r-1}}
h\pr{X_{i_1},\dots,X_{i_{r-1}},s}         }>x^{1-1/r}y^{1/r}\pr{u/4}^{1/2}     }\mathrm{d}\PP_X\pr{s}.
\end{multline}
We  bound the probability inside the integral for each fixed $s$. To this aim, we 
use the assumption that the assertion $A\pr{r-1}$ is true. Define for a fixed $s\in S$ 
the function $\widetilde{h}\colon S^{r-1}\to \R$ by 

\begin{equation}
 \widetilde{h}\pr{x_1,\dots,x_{r-1}}
=h\pr{x_1,\dots,x_{r-1},s}.
\end{equation}
Then $\widetilde{h}$ satisfied the condition \eqref{eq:condition_de_degeneration} 
with $r$ replaced by $r-1$. Let 
\begin{equation}
\widetilde{x}:= x^{1-1/r}y^{1/r}\pr{u/4}^{1/2},       \quad \widetilde{y}:= y\pr{u/4}^{1/2}
\pr{1+2\ln\pr{ u}}^{-\pr{r-1}\frac{1}2}.
\end{equation}
Since $x/y\geq 3^r$, it follows that   $\widetilde{x}/\widetilde{y}=\pr{x/y}^{1-1/r   }\pr{1+
2\ln\pr{ u}}^{\pr{r-1}\frac{1}2}
\geq \pr{x/y}^{1-1/r   }\geq 3^{r-1}$ hence we are in position to apply $A\pr{r-1}$, which 
gives
\begin{multline}
\PP\ens{ \frac 1{N^{\frac{r-1}2}      }     
\max_{r-1\leq k\leq n} \abs{ \sum_{\pr{i_1,\dots,i_{r-1}}\in I_k^{r-1}}
h\pr{X_{i_1},\dots,X_{i_{r-1}},s}         }>x^{1-1/r}y^{1/r}\pr{u/4}^{1/2}   }
\\ \leq  A_{r-1}\exp\pr{-\frac 12\pr{\frac{\widetilde{x}}{\widetilde{y}}}^{ \frac{2}{ r-1    }}}
\\+B_{r-1}\int_{1}^{+\infty}
\PP\ens{\abs{ \widetilde{h}\pr{X_1,\dots,X_{r-1}}}> \widetilde{y} v^{1/2}C_{r-1}    } 
\pr{1+\ln\pr{ v}}^{q_{r-1}}\mathrm dv.
\end{multline}
Replacing $\widetilde{x}$, $\widetilde{y}$ and $\widetilde{h}$ by their corresponding 
expression and integrating over $S$ with respect to the law of $X$ gives in view of 
\eqref{eq:preparation_de_l_etape_de_recurrence} that 
\begin{multline}\label{eq:control_queue_de_Y}
\PP\ens{ Y>x^{1-1/r}y^{1/r}\pr{u/4}^{1/p}} 
\leq A_{r-1}\exp\pr{-\frac 12\pr{  \pr{\frac xy}^{\frac{2}{r }}
 \pr{1+2\ln\pr{ u}} }}\\
 +B_{r-1}\int_{1 }^{+\infty}
\PP\ens{\abs{ h\pr{X_1,\dots,X_{r-1},X}}> y\pr{u/4}^{1/2}
\pr{1+\ln\pr{ u}}^{-\pr{r-1}\frac{ 1}2} v^{1/2}C_{r-1}       } 
\pr{1+\ln\pr{ v}}^{q_{r-1}}\mathrm dv.
\end{multline}
We derive in 
view of \eqref{eq:se_ramener_au_cas_r-1} and the fact that 
$x/y>3^r$ 
that 
\begin{equation}
\int_{1 }^{+\infty}
\exp\pr{-\frac 12\pr{  \pr{\frac xy}^{\frac{2}{r }}
 \pr{1+2\ln\pr{ u} } }}\mathrm du 
\leq \exp\pr{-\frac 12  \pr{\frac xy}^{\frac{2}{r }}}
\int_{1}^{+\infty}
 t^{-6}\mathrm dt
\end{equation} 
hence computing the integral,  
\begin{equation}
\int_{1 }^{+\infty}
\exp\pr{-\frac 12\pr{  \pr{\frac xy}^{\frac{2}{r }}
 \pr{1+2\ln\pr{ u} } }}\mathrm du \leq  
 \frac{1}{2 }  \exp\pr{-\frac 12 \pr{\frac xy}^{\frac{2}{r }}}.
\end{equation}
Integrating \eqref{eq:control_queue_de_Y} with respect to $u$ on 
$\pr{1 ,+\infty}$, we obtain 
\begin{multline}\label{eq:apres_utilisation_recurrence}
\PP\ens{\max_{r\leq n\leq N}\abs{U_n}>N^{r/2}x}
\leq \pr{2+A_{r-1}/2}\exp\pr{-\frac 12\pr{\frac{x}{y}}^{2}  } \\
 +  B_{r-1}\int_{1 }^{+\infty}\int_{1 }^{+\infty}
 \PP\ens{\abs{ h\pr{X_1,\dots,X_{r-1},X_r}}>  y\pr{u/4}^{1/2}
\pr{1+\ln\pr{u}}^{-\pr{r-1}\frac{ 1}2} v^{1/2}C_{r-1}       } 
\pr{1+\ln\pr{ v}}^{q_{r-1}}\mathrm dv\mathrm{d}u.
\end{multline}
In order to control the last term, we will make a use of the following lemma. 

\begin{Lemma}\label{lem:integrale_double_avec_log}
Let $X$ be a non-negative random variable and let $q,q'$ be   non-negative
numbers. 
Define
\begin{equation}
 \kappa_q=\begin{cases}
           1&\mbox{ if }q\leq 1\\
           e^{q-1}/q^q&\mbox{ if }q> 1.
          \end{cases}
\end{equation}

Then 
\begin{multline}\label{eq:resultat_lemme_double_integrale}
\int_{1 }^{+\infty}\int_{1 }^{+\infty}
 \PP\ens{X> u 
\pr{1+\ln\pr{ u}}^{-  q_1  } v     } 
\pr{1+\ln\pr{ v}}^{q_2}\mathrm dv\mathrm{d}u\\ 
\leq \frac{2^{q_1}}{\kappa{q_1} }\pr{1+\ln\frac{\kappa_{2q_1}}{\kappa_{q_1}}}^{q_1}
\int_1^{+\infty} \PP\ens{X>t/\kappa_{q_1}    }\pr{1+\ln t}^{q_1+q_2+1}dt.
\end{multline} 
\end{Lemma}

\begin{proof}
Define the function $g_q\colon u\mapsto u/\pr{1+\ln u}^{-q}$ 
for $u\geq 1$. Then 
$g'_q\pr{u}= \pr{1+\ln u}^{-q}-uq \pr{1+\ln u}^{-q-1}\frac 1u
= \pr{1+\ln u}^{-q-1}\pr{1+\ln u-q}$ hence the 
function $g_q$ reaches its minimum at $u=e^{q-1}$. In particular, 
\begin{equation}\label{eq:controle_fct_gq}
 \pr{1+\ln u}^q\leq \kappa_q u, u\geq 1,
\end{equation}
where 
\begin{equation}
 \kappa_q=\begin{cases}
           1&\mbox{ if }q\leq 1\\
           e^{q-1}/q^q&\mbox{ if }q> 1.
          \end{cases}
\end{equation}

We  do the substitution $w=u 
\pr{1+\ln\pr{ u}}^{-  q_1 } v $ for a fixed $u$. Then 
\begin{multline}\label{eq:simplication_double_integral}
 \int_{1 }^{+\infty}\int_{1 }^{+\infty}
 \PP\ens{X> u 
\pr{1+\ln\pr{ u}}^{-  q_1 } v     } 
\pr{1+\ln\pr{ v}}^{q_2}\mathrm dv\mathrm{d}u\\= 
\int_{1 }^{+\infty}\int_{g_{q_1}\pr{u}   }^{+\infty}
 \PP\ens{X> w    } 
\pr{1+\ln\pr{w /g_{q_1}\pr{u}      }}^{q_2}  
g_{q_1}\pr{u}^{-1}\mathrm dw\mathrm{d}u.
\end{multline}

Since $g_{q_1}\pr{u}^{-1}\leq \kappa_q$, it follows that 
\begin{multline}\label{eq:simplication_double_integral_bis}
  \int_{1 }^{+\infty}\int_{1 }^{+\infty}
 \PP\ens{X> u 
\pr{1+\ln\pr{ u}}^{-  q_1 } v     } 
\pr{1+\ln\pr{ v}}^{q_2}\mathrm dv\mathrm{d}u\\
\leq 
\int_{0 }^{+\infty}\PP\ens{X>w}
I\pr{w}\pr{1+\ln\pr{w\kappa_{q_1}      }}^{q_2}  dw
\end{multline}
where 
\begin{equation}
 I\pr{w}=\int_1^{+\infty}
 \mathbf{1}_{\pr{0,w}}\pr{g_{q_1}\pr{u}   } 
 g_{q_1}\pr{u}^{-1}
 du.
\end{equation}
Observe that if $w< 1/\kappa_{q_1}$, then 
$I\pr{w}=0$. Assume now that $w \geq 1/\kappa_{q_1}$. 
Using \eqref{eq:controle_fct_gq} with $q=2q_1 $, we get that 
$g_{2q_1 }\pr{u}\geq \kappa_{2q_1}^{-1}$ hence 
$g_{q_1}\pr{u}\geq \kappa_{2q_1}^{-1}
\sqrt u$ and it follows that 
\begin{equation}
\mathbf{1}_{\pr{0,w}}\pr{g_{q_1}\pr{u}   } \leq 
\mathbf{1}_{\pr{0,\kappa_{2q_1}^2w^2}}\pr{u}     
\end{equation}
hence 
\begin{equation}
 I\pr{w}\leq \frac 1{q_1+1}
 \pr{1+2\ln\pr{w\kappa_{2q_1}}  }^{q_1+1}
\end{equation}
hence by  \eqref{eq:simplication_double_integral_bis}, 
\begin{multline}
   \int_{1 }^{+\infty}\int_{1 }^{+\infty}
 \PP\ens{X> u 
\pr{1+\ln\pr{ u}}^{-  q_1 } v     } 
\pr{1+\ln\pr{ v}}^{q_2}\mathrm dv\mathrm{d}u \\
\leq 
\frac 1{q_1+1}
\int_{1/\kappa_{q_1}}^{+\infty}\PP\ens{X>w}
 \pr{1+2\ln\pr{w\kappa_{2q_1}}  }^{q_1+1}\pr{1+\ln\pr{w\kappa_{q_1}      }}^{q_2}  dw\\
 \leq 
2^{q_1} 
\int_{1/\kappa_{q_1}}^{+\infty}\PP\ens{X>w}
 \pr{1+ \ln\pr{w\kappa_{2q_1}}  }^{q_1+q_2+1} \mathrm{d}w.
\end{multline}

We get \eqref{eq:resultat_lemme_double_integrale}
after the substitution 
$t=\kappa_{q_1}w$ and the elementary inequality $
\pr{1+\ln\pr{at}}^q\leq \pr{1+\ln a}^q\pr{1+\ln t}^q$.
 This ends the proof of Lemma~\ref{lem:integrale_double_avec_log}.
\end{proof}

To conclude the proof of 
Theorem~\ref{thm:inegalite_deviation_U_stats}, we 
apply \eqref{eq:apres_utilisation_recurrence} in the following setting:
\begin{itemize}
\item $X:=4\abs{h\pr{X_1,\dots,X_r}}^2 C_{r-1}^{-2}y^{-2}$;
\item $q_1=  r-1 $;
\item $q_2=q_{r-1}$. 
\end{itemize}

\subsubsection{Proof of Corollary~\ref{cor:extension_inegalite_de_deviation}}

We start from \eqref{eq:decomposition_de_Hoeffding_deg_ordre_i_1}. 
We get that for each $r\leq n\leq N$, 
\begin{equation}
 \frac{1}{N^{r-i/2 }}\abs{U_{r,n}\pr{h}}
 \leq \sum_{k=i}^r\frac 1{\pr{r-k}!}N^{-k+i/2}
 \abs{U_{k,n}\pr{h_k}}
\end{equation}
and taking the maximum over $n\in\ens{r,\dots,N}$ 
gives 
\begin{equation}
 \frac{1}{N^{r-i/2}}\max_{r\leq n\leq N}
 \abs{U_{r,n}\pr{h}}
 \leq \sum_{k=i}^r\frac 1{\pr{r-k}!}N^{-k+i/2}
 \max_{k\leq n\leq N}
 \abs{U_{k,n}\pr{h_k}}.
\end{equation}
It follows that 
\begin{equation}
 \PP\ens{\max_{r\leq n\leq N}\abs{U_n}>N^{r-i/2}x} 
\leq 
\sum_{k=i}^r \PP\ens{
\frac 1{\pr{r-k}!}N^{-k+i/2}
 \max_{k\leq n\leq N}
 \abs{U_{k,n}\pr{h_k}}>x/r
}.
\end{equation}
For all $k\in\ens{i,\dots,r}$, the equality 
\begin{equation}
 \PP\ens{
\frac 1{\pr{r-k}!}N^{-k+i/2}
 \max_{k\leq n\leq N}
 \abs{U_{k,n}\pr{h_k}}>x/r
}= 
\PP\ens{
 \frac 1{N^{k/2}}\max_{k\leq n\leq N}
 \abs{U_{k,n}\pr{h_k}}>\widetilde{x}   }
\end{equation}
holds, 
where $\widetilde{x}=N^{ \pr{k-i}/2}x\pr{r-k}!/r$.

We can therefore apply Theorem~\ref{thm:inegalite_deviation_U_stats} 
in the setting 
\begin{itemize}
 \item $\widetilde{x}=N^{ \pr{k-i}/2}x\pr{r-k}!/r$; 
 \item $\widetilde{y}=N^{ \pr{k-i}/2}y\pr{r-k}!/r$ and 
 \item $\widetilde{r}=k$
\end{itemize}
to get the wanted result.

\subsection{Proof of Corollary~\ref{cor:inegalite_moments_exp}}

We first use Corollary~\ref{cor:extension_inegalite_de_deviation} with a
 $y<3^{-r}x$ that will be specified later. Observe that in view of Markov's inequality, 
\begin{equation}
 \int_{1 }^{+\infty}
\PP\ens{\abs{Y_k}> y 
N^{\pr{k-i}\frac{1}{2}}u ^{1/2} C_{r,2}   } 
\pr{1+\ln\pr{ u}}^{q_{k,2}}\mathrm du
\leq I\pr{y}\E{\exp\pr{ \abs{Y_k}^\gamma  }},
\end{equation}
where 
\begin{equation}\label{eq:definition_de_I_y}
I\pr{y}:= 
 \int_{1 }^{+\infty}
\exp\pr{-  \pr{y 
N^{\pr{k-i}\frac{ 1}{2}}u ^{1/2} C_{r,2}}^\gamma  } 
\pr{1+\ln\pr{ u}}^{q_{k,2}}\mathrm du.
\end{equation}
Now, noticing that $\E{\exp\pr{  \abs{Y_k}^\gamma  }}\leq c_\gamma
\E{\exp\pr{ \abs{h\pr{X_1,\dots,X_r}}^\gamma  }}$, it is sufficient to 
treat the case $k=i$. 

Doing the substitution $v=u^{\gamma/2}-1$, we obtain 
\begin{equation}
I\pr{y}= \frac 2{\gamma}\int_{1 }^{+\infty}
\exp\pr{- C_{r,2}^\gamma y^\gamma  \pr{v+1}  }
\pr{1+\ln\pr{ \pr{v+1}^{2/\gamma}    }}^{q_{r,2}}\pr{1+v}^{\frac{2}{\gamma}-1}
\mathrm dv.
\end{equation}
Assume that $y$ is such that 
\begin{equation}\label{eq:condition_sur_y}
y^\gamma C_{r,2}^\gamma \geq 1.
\end{equation}

Then 
\begin{equation}
I\pr{y}\leq \exp\pr{-R C_{r,2}^\gamma y^\gamma    }
\cdot \frac 2{\gamma}\int_{1 }^{+\infty}
\exp\pr{- v   }
\pr{1+\ln\pr{ \pr{v+1}^{2/\gamma}    }}^{q_{r,2}}\pr{1+v}^{\frac{2}{\gamma}-1}
\mathrm dv.
\end{equation}
Now, we choose $y<x3^{-r}$ satisfying \eqref{eq:condition_sur_y} such that 
\begin{equation}
 \frac 12\pr{\frac xy}^{ 2/i    }=  C_{r,2}^\gamma y^\gamma,
\end{equation}
This leads to the choice 
\begin{equation}
y:=x^{\frac{2}{2+i\gamma}}2^{-\frac{1}{\gamma+2/i}}C_{r,2}^{-\frac{\gamma}{\gamma+2/i}}.
\end{equation}
Due to the definition of $x_{r,i\gamma}$, the inequalities $x/y\geq 3^r$ and 
\eqref{eq:condition_sur_y} hold. This choice of $y$ combined with the observation that 
$1\leq \E{\exp\pr{ \abs{h\pr{X_1,\dots,X_r} }^\gamma   }}$ gives 
\eqref{eq:inegalite_Ustats_moments_expo}, which ends the proof of 
Corollary~\ref{cor:inegalite_moments_exp}.

\subsection{Proof of Theorem~\ref{thm_Baum_Katz_Ustats}}

Let $\eps>0$. Observe that for all positive $C$, 
 \begin{equation}
 \PP\ens{\max_{r\leq n\leq 2^N}\abs{U_n}>\eps 2^{N\alpha}}=
 \PP\ens{\max_{r\leq n\leq 2^N}\abs{\frac{CU_n}{\eps  }}>2^{N\pr{r-i/2}}
  C2^{N\pr{\alpha -\pr{r-i/2} }  }
 }.
 \end{equation}
We now choose $C$ such that $C^{\frac{2\gamma}{i\gamma+2}}B_{r,\gamma}\geq 1$, 
where $B_{r,\gamma}$ is like in Corollary~\ref{cor:inegalite_moments_exp}. We then apply for 
  Corollary~\ref{cor:inegalite_moments_exp} in the settting $\widetilde{h}
:= Ch/\eps$, $x =  C2^{N\pr{\alpha -\pr{r-i/2} }}  $ and for  $N$ 
such that $C2^{N\pr{\alpha -\pr{r-i/2} }}\geq x_{r,i,\gamma}$, with $\widetilde{N}=2^N$.

\subsection{Proof of Theorem~\ref{thm:large_deviation_Ustats}}

We apply Corollary~\ref{cor:extension_inegalite_de_deviation} with $x$ replaced by 
$xN^{i/2}$. This gives 
\begin{multline} \label{eq:step_large_deviation}
\PP\ens{\max_{r\leq n\leq N}\abs{U_n}>N^{r}x} 
\leq A_{r}\exp\pr{- \frac 12\pr{\frac{xN^{i/2}}y}^{ \frac{2}{i   }}}
\\+B_{r}\sum_{k=i}^r\int_{1 }^{+\infty}
\PP\ens{\abs{\E{h\pr{X_1,\dots,X_r}\mid X_1,\dots,X_k   }}> y 
N^{\pr{k-i}\frac{ 1}{2}}u ^{1/2} C_{r}   } 
\pr{1+\ln\pr{ u}}^{q_{k}}\mathrm du.
\end{multline}
Using \eqref{eq:inegalites_queues_generale}, we notice that 
$\sup_{t>0}\exp\pr{t^\gamma}\PP\ens{  \abs{\E{h\pr{X_1,\dots,X_r}\mid X_1,\dots,X_k   }}>t    }
\leq \kappa \cdot M$, where $\kappa$ depends only on $\gamma$. Therefore, we will only focus on
the term associated to $k=i$ in the right hand side of \eqref{eq:step_large_deviation}. After having 
used the condition on the tail, we obtain 
\begin{multline}
\PP\ens{\max_{r\leq n\leq N}\abs{U_n}>N^{r}x} 
\leq A_{r}\exp\pr{- \frac 12\pr{\frac{xN^{i/2}}y}^{ \frac{2}{i   }}}
\\+B_{r,\gamma}M \int_{1 }^{+\infty}
 \exp\pr{-\pr{ y 
 \frac{ 1}{2}u ^{1/2} C_{r}   }^{\gamma}}
\pr{1+\ln\pr{ u}}^{q_{k}}\mathrm du.
\end{multline}
Then using similar arguments as in the control of $I\pr{y}$ defined by 
 \eqref{eq:definition_de_I_y}, we get 
 \begin{equation}
\PP\ens{\max_{r\leq n\leq N}\abs{U_n}>N^{r}x} 
\leq A_{r}\exp\pr{- \frac 12\pr{\frac{xN^{i/2}}y}^{ \frac{2}{i   }}}
\\+B'_{r,\gamma}M 
 \exp\pr{-C_{r,\gamma}  y^{\gamma}}.
\end{equation}

We conclude the proof by choosing $y=N^{\frac{i}{2+i\gamma}}x^{\frac{2 }{2+i\gamma} }$. 

\subsection{Proof of the results of Subsection~\ref{subsec:WIP_Holder}}

\begin{proof}[Proof of Proposition~\ref{prop:critere_de_tension}]
First, by using Theorem~3 in \cite{MR2054586} (which is a consequence of 
the Schauder decomposition of $\Hca_\rho^o$ and the tightness 
criterion given in \cite{MR1736910}), the following condition is sufficient for tightness 
of a sequence of processes $\pr{\xi_n}_{n\geq 1}$ in $\Hca_\rho^o$:
\begin{equation}
\forall\eps>0, \lim_{J\to +\infty}\limsup_{n\to +\infty}
\PP\ens{\sup_{j\geq J}\max_{t\in D_j} \abs{\lambda_{j,t }\pr{\xi_n}}/\rho\pr{2^{-j}}>\eps  }=0,
\end{equation}
where for $j\geq $, $D_j=\ens{\pr{2k+1}2^{-j},0\leq k\leq 2^j-1}$ and 
\begin{equation}
\lambda_{j,t}\pr{x}=x\pr{t}-\frac 12\pr{x\pr{t+2^{-j}}-x\pr{t-2^{-j}}   }, x\in \Hca_\rho^o, t\in D_j.
\end{equation}
Since 
\begin{equation}
\abs{\lambda_{j,t }\pr{\xi_n}}\leq \frac 12\abs{x\pr{t} -   x\pr{t+2^{-j}}}   +
\frac 12\abs{x\pr{t} -   x\pr{t-2^{-j}}},
\end{equation}
we infer that 
\begin{equation}
\max_{t\in D_j} \abs{\lambda_{j,t }\pr{\xi_n}}\leq 
\max_{1\leq \ell \leq 2^j}\abs{\xi_n\pr{ \ell 2^{-j}      } -   \xi_n\pr{\pr{\ell-1}  2^{-j}}}
\end{equation}
and we are therefore reduced to prove that 
\begin{equation}
\forall\eps>0, \lim_{J\to +\infty}\limsup_{n\to +\infty}
\PP\ens{\sup_{j\geq J}  \max_{1\leq \ell \leq 2^j}
\abs{W_n\pr{ \ell 2^{-j}      } -  W_n\pr{\pr{\ell-1}  2^{-j}}}/\rho\pr{2^{-j}}>\eps  }=0.
\end{equation}
We will now control the differences $\abs{W_n\pr{ \ell 2^{-j}      } -  W_n\pr{\pr{\ell-1}  2^{-j}}}$ 
by exploiting the fact that the graph of $W_n$ is a polygonal line. 
Fix an integer $n$ and define the interval $I_k:=\left[\pr{k-1}/n,k/n\right]$, $1\leq k\leq n$. Define 
\begin{equation}
R_n:=\max_{1\leq k\leq n}\abs{W_n\pr{k/n}-W_n\pr{\pr{k-1}/n}}.
\end{equation}
Let $0\leq s<t\leq 1$. 
\begin{enumerate}
\item There exists a $k\in\ens{1,\dots,n}$ such that $s,t\in I_k$. Since on $I_k$, 
$W_n$ is affine, we derive that 
\begin{align}
\abs{W_n\pr{t}-W_n\pr{s}}&\leq n\pr{t-s}\abs{W_n\pr{k/n}-W_n\pr{\pr{k-1}/n}}\\ 
&\leq n\pr{t-s}R_n.
\end{align}
\item There exists a $k\in\ens{1,\dots,n-1}$ such that $s\in I_k$ and $t\in I_{k+1}$. Starting 
from 
\begin{equation}
\abs{W_n\pr{t}-W_n\pr{s}}\leq \abs{W_n\pr{t}-W_n\pr{k/n}}+\abs{W_n\pr{k/n}-W_n\pr{s}}
\end{equation}
and applying the reasoning of the first case to treat the two terms, we get 
\begin{equation}
\abs{W_n\pr{t}-W_n\pr{s}}\leq n\pr{t-s}R_n.
\end{equation}
\item There exists a $k\in\ens{1,\dots,n}$ such that $s\in I_k$ and $j\in\ens{k+2,\dots,n}$ such that 
$t\in I_j$. We start from 
\begin{equation}
\abs{W_n\pr{t}-W_n\pr{s}}\leq \abs{W_n\pr{t}-W_n\pr{\frac{j-1}n}}+
\abs{W_n\pr{\frac{j-1}n}-W_n\pr{\frac{k}{n}}}+\abs{W_n\pr{\frac{k}{n}}-W_n\pr{s}}.
\end{equation}
For the first and third terms of the right hand side, we use the reasoning of case 1 to get that 
their contribution does not exceed $2 R_n$. The second term is $\abs{S_{j-1}-S_k}/a_n$ 
hence 
\begin{equation}
\abs{W_n\pr{t}-W_n\pr{s}}\leq 3R_n+\frac{\abs{S_{\left[nt\right]       }-
S_{\left[ns\right]       }}}{a_n} .
\end{equation}
\end{enumerate}

Suppose that $j\geq [\log_2n]+1$ and let $t=\ell 2^{-j}$ and $s=\pr{\ell-1}2^{-j}$. Then 
$t-s=2^{-j}<1/n$ hence only the first two cases  are possible. Consequently, 
\begin{equation}
\sup_{j\geq [\log_2n]+1}\max_{1\leq \ell\leq 2^j}\abs{W_n\pr{ \ell 2^{-j}      } -  W_n\pr{\pr{\ell-1}  2^{-j}}}/\rho\pr{2^{-j}}\leq \sup_{j\geq [\log_2n]+1}n 2^{-j}R_n/\rho\pr{2^{-j}}.
\end{equation}
Now, exploiting the fact that $\rho\pr{u}=u^\alpha L\pr{1/u}$, we infer that for some constant $C$ 
depending only on $\alpha$ and $L$, 
\begin{equation}
\sup_{j\geq [\log_2n]+1}\max_{1\leq \ell\leq 2^j}\abs{W_n\pr{ \ell 2^{-j}      } -  W_n\pr{\pr{\ell-1}  2^{-j}}}/\rho\pr{2^{-j}}\leq C R_n/\rho\pr{1/n}.
\end{equation}

Let now $j\in\ens{J,\dots,[\log_2n]}$. This time, with the choices $t=\ell 2^{-j}$ and $s=\pr{\ell-1}2^{-j}$, 
the third case applies hence 
\begin{multline}
   \max_{J\leq j\leq [\log_2n]}   
   \max_{1\leq \ell\leq 2^j}\abs{W_n\pr{ \ell 2^{-j}      } -  W_n\pr{\pr{\ell-1}  2^{-j}}}/\rho\pr{2^{-j}}
\\   
   \leq 
   3 \max_{J\leq j\leq [\log_2n]}\rho\pr{2^{-j}}^{-1}     R_n+\max_{J\leq j\leq [\log_2n]} 
   \rho\pr{2^{-j}}^{-1}   
   \max_{1\leq \ell\leq 2^j}\frac 1{a_n}\abs{S_{\left[n\ell 2^{-j}\right] } -
   S_{\left[n\pr{\ell-1} 2^{-j}\right] }  }.
\end{multline}
In total, we got that for a constant $C$ 
depending only on $\rho$, 
\begin{multline}
   \sup_{j\geq J}   
   \max_{1\leq \ell\leq 2^j}\abs{W_n\pr{ \ell 2^{-j}      } -  W_n\pr{\pr{\ell-1}  2^{-j}}}/\rho\pr{2^{-j}}
\\   
   \leq 
  C\max_{J\leq j\leq [\log_2n]} \rho\pr{2^{-j}}^{-1} R_n+\max_{J\leq j\leq [\log_2n]}   \rho\pr{2^{-j}}^{-1} 
   \max_{1\leq \ell\leq 2^j}\frac 1{a_n}\abs{S_{\left[n\ell 2^{-j}\right] } -
   S_{\left[n\pr{\ell-1} 2^{-j}\right] }  }.
\end{multline}
Since \eqref{eq:tightness_criterion} guarantees that 
\begin{equation}
\lim_{J\to +\infty}\limsup_{n\to +\infty}\PP\ens{
\max_{J\leq j\leq [\log_2n]}    \rho\pr{2^{-j}}^{-1} 
   \max_{1\leq \ell\leq 2^j}\frac 1{a_n}\abs{S_{\left[n\ell 2^{-j}\right] } -
   S_{\left[n\pr{\ell-1} 2^{-j}\right] }  }>\eps
}=0,
\end{equation} 
it remains to check that the sequence $\pr{\max_{1\leq j\leq [\log_2n]} \rho\pr{2^{-j}}^{-1} R_n}_{n\geq 1}$ converges to $0$ in probability. 
Due to the construction of $W_n$ and the assumptions on the sequence $\pr{a_n}_{n\geq 1}$ and $\rho$, 
it suffices to check that the convergence in probability of $\pr{R_{2^N}/\rho\pr{2^{-N}}}_{N\geq 1}$. 
To this aim, notice that \eqref{eq:tightness_criterion} implies (by considering $n=2^N$) that 
for each positive $\eps$, 
\begin{equation}
\lim_{N\to +\infty} \sum_{k=0}^{2^N-1}
   \PP\ens{\abs{S_{ [2^N\pr{k+1}2^{-N}    ]   }-S_{ [2^Nk2^{-N}]   }  }    >a_{2^N} 
   \eps\rho\pr{2^{-N}}   }=0,
\end{equation}
which implies the convergence in probability to $0$ of $R_{2^N}/\rho\pr{2^{-N}}$. 
This ends the proof of Proposition~\ref{prop:critere_de_tension}.
\end{proof}

\begin{proof}[Proof of Proposition~\ref{prop:deviation_accroissements}]
We start from the equalities 
\begin{align}
\frac 1{\sqrt{n_2-n_1}n_2^{\frac{r-1}2}}\pr{U_{n_2}-U_{n_1}}&= 
\frac 1{\sqrt{n_2-n_1}n_2^{\frac{r-1}2}}\sum_{\substack{1\leq i_1<\dots<i_r\leq n_2 \\ n_1+1\leq i_r\leq n_2 }}
h\pr{X_{i_1},\dots,X_{i_r}}\\ 
&=\frac 1{\sqrt{n_2-n_1}}\sum_{j=n_1+1}^{n_2}D_j,
\end{align}
where 
\begin{equation}
D_j=\frac{1}{n_2^{\frac{r-1}2}}\sum_{1\leq i_1<\dots<i_{r-1}<j}h\pr{X_{i_1},\dots,X_{i_{r-1}},X_j}.
\end{equation}
Define the random variable $Y$ as 
\begin{equation}
\frac 1{n_2^{\pr{r-1}/2}} \max_{r-1\leq j\leq n_2}
\abs{\sum_{1\leq i_1<\dots<i_{r-1}<j} h\pr{X_{i_1},\dots,X_{i_{r-1}},X}       },
\end{equation}
where $X$ is independent of the sequence $\pr{X_i}_{i\geq 1}$ and has the same 
distribution as $X$. Then in the same way as in the proof of 
Theorem~\ref{thm:inegalite_deviation_U_stats}, we can prove that $\pr{D_j}_{j=n_1+1}^{n_2}$ 
is a martingale differences sequence and that $D_j^2\conv Y^2$ for all $j\in\ens{n_1+1,
\dots,n_2}$. Applying Proposition~\ref{prop:inegalite_deviation_martingales_dominee}, we derive that 
\begin{multline}
\PP\ens{\frac 1{\sqrt{n_2-n_1}n_2^{\frac{r-1}2}}\abs{U_{n_2}-U_{n_1}}>x}\\
\leq 2\exp\pr{-\frac 12 \frac{x^2}{y^2}}+2\int_1^{+\infty}
\PP\ens{ \frac 1{n_2^{\pr{r-1}/2}} \max_{r-1\leq j\leq n_2}
\abs{\sum_{1\leq i_1<\dots<i_{r-1}<j} h\pr{X_{i_1},\dots,X_{i_{r-1}},X}       }>y\sqrt{u}/2   }\mathrm{d}u.
\end{multline}
Then we treat the last integral in the following way: we integrate with respect to the law of $X$, 
apply Theorem~\ref{thm:inegalite_deviation_U_stats} to the $U$-statistics of order $r-1$ and 
rearrange the integrals. 
\end{proof}

\begin{proof}[Proof of Theorem~\ref{thm:PI_Holderien}]
The convergence of the finite dimensional distributions follows from Corollary~1 
in  \cite{MR740907} and the fact that for a fixed $t$, the contribution of 
$n^{-i/2}\pr{ nt-[nt]}\pr{U_{[nt]+1}-U_{[nt]}}$ is negligible.

It remain to prove tightness in $\Hca_\rho^o$. To this aim, we apply the 
Hoeffding's decomposition \eqref{eq:decomposition_de_Hoeffding_deg_ordre_i_1} and it suffices 
to show the following. 
\begin{Lemma}\label{lem:tension_termes_de_la_decomposition_de_Hoeffding}
Let  $r\geq 1$ be an integer,  $\pr{S,\Sca}$ be measurable space, $h\colon 
S^r\to \R$ be a symmetric measurable function (with $S^r$ induced with the product 
$\sigma$-algebra) and let $\pr{X_i}_{i\geq 1}$ be an i.i.d. sequence of 
$S$-valued random variables. Assume that $h$ is degenerated. Let $\rho\in\Rca_r$.  
Suppose that \eqref{eq:condition_suffisante_WIP} holds.
 Then 
\begin{equation}
\pr{\frac 1{n^{r/2}}\pr{U_{[nt]}-\pr{nt-[nt]} \pr{U_{[nt]+1} -U_{[nt]}  }  } }_{n\geq 1}
\end{equation}
is tight in $\Hca_\rho^o$.
\end{Lemma}
Let us first show how Lemma~\ref{lem:tension_termes_de_la_decomposition_de_Hoeffding} helps 
to conclude. After having applied the Hoeffding decomposition, the considered partial sum is 
a sum of partial sum process defined like in \eqref{eq:definition_processus_sommes_partielles} but 
with $U$-statistics of lower order and overall degenerated, with a stronger normalization than 
$n^{k/2}$ for each term of the sum in \eqref{eq:decomposition_de_Hoeffding_deg_ordre_i_1}. 
This shows tightness of the initially considered process and concludes the proof.

\begin{proof}[Proof of Lemma~\ref{lem:tension_termes_de_la_decomposition_de_Hoeffding}]
By an application of Proposition~\ref{prop:critere_de_tension}, it suffices to show that for 
all positive $\varepsilon$, 
\begin{equation}
\lim_{J\to +\infty}\limsup_{n\to +\infty}\sum_{j=J}^{[\log_2n]}\sum_{k=0}^{2^j-1}
   \PP\ens{\abs{U_{ [n\pr{k+1}2^{-j}    ]   }-U_{ [nk2^{-j}]   }  }   >n^{r/2} \eps\rho\pr{2^{-j}}   }=0.
\end{equation}
In order to control the involved probabilities, we apply Proposition~\ref{prop:deviation_accroissements} 
for fixed $n$, $J$, $j\in \ens{J,\dots,[\log_2 n]}$ and $k\in\ens{0,\dots,2^j-1}$ 
in the following setting: $n_1=[nk2^{-j}] $, $n_2=[n\pr{k+1}2^{-j}    ] $ and 
\begin{equation}
x:= \frac{n^{r/2} \eps\rho\pr{2^{-j}}   }{  \sqrt{[n\pr{k+1}2^{-j}    ] -[nk2^{-j}]}  [n\pr{k+1}2^{-j}    ]^{\frac{r-1}{2}  }   }.
\end{equation}
Then we choose $y$ such that 
\begin{equation}
\pr{\frac xy}^{2/r}=3 j  \ln 2.
\end{equation}
Observing that $x\geq c 2^{j/2}\rho\pr{2^{-j}} $, we derive that 
\begin{multline}
 \PP\ens{\abs{U_{ [n\pr{k+1}2^{-j}    ]   }-U_{ [nk2^{-j}]   }  }   >n^{r/2} \eps\rho\pr{2^{-j}}   }\\
 \leq A_r 2^{-3j/2}+
 B_r\int_1^{+\infty}\PP\ens{\abs{h\pr{X_1,\dots,X_r}}> c2^{j/2}\rho\pr{2^{-j}}j^{-r/2} u^{1/2}}
 \pr{1+\log u}^{q_r}du.
\end{multline}
We conclude by summing over $j\geq J$  and exploiting the convergence of the involved series.

\end{proof}

Now, it remains to show that conditions \eqref{eq:cond_suff_WIP_talpha} (respectively
\eqref{eq:cond_suff_WIP_t12}) are sufficient in the case $\rho\pr{t}=t^\alpha$ (respectively 
$\rho\pr{t}=t^{1/2}\pr{\log\pr{c/t}}^\beta$). 

When $\rho\pr{t}=t^\alpha$, we first use the fact that if $a$ and $b$ are two positive real 
numbers and $Z$ is a non-negative random variable, then 
\begin{equation}
\sum_{j\geq 1}2^j\PP\ens{ Z >\frac{2^{ja}}{\pr{1+j}^b}  }\leq 
C_{a,b}\E{Z^{1/a}\pr{1+Z}^b\mathbf 1\ens{Z>1}}.
\end{equation}
This can be seen by cutting the tail probability on a sum of probabibilities that 
$Z$ lies in $\left(c_j,c_{j+1}\right]$, where $c_j=\frac{2^{ja}}{\pr{1+j}^b} $ then 
switching the sums. Applying this to $Z=Y/v^{1/2}$ for a fixed $v$, $a=1/2-\alpha$ and 
$b=r/2$, bounding the logarithms by $\pr{1+\log Y}^{r/2 }= $ 
and accounting $1/\pr{1/2-\alpha}$ gives \eqref{eq:condition_suffisante_WIP}.

When $\rho\pr{t}=t^{1/2}\pr{\log\pr{c/t}}^\beta$, this follows from similar estimates as in 
the proof of Corollary~\ref{cor:inegalite_moments_exp}.
This concludes the proof of Theorem~\ref{thm:PI_Holderien}.
\end{proof}

\def\polhk\#1{\setbox0=\hbox{\#1}{{\o}oalign{\hidewidth
  \lower1.5ex\hbox{`}\hidewidth\crcr\unhbox0}}}\def\cprime{$'$}
  \def\polhk#1{\setbox0=\hbox{#1}{\ooalign{\hidewidth
  \lower1.5ex\hbox{`}\hidewidth\crcr\unhbox0}}} \def\cprime{$'$}
\providecommand{\bysame}{\leavevmode\hbox to3em{\hrulefill}\thinspace}
\providecommand{\MR}{\relax\ifhmode\unskip\space\fi MR }
\providecommand{\MRhref}[2]{%
  \href{http://www.ams.org/mathscinet-getitem?mr=#1}{#2}
}
\providecommand{\href}[2]{#2}


\begin{thebibliography}{10}

\bibitem{MR1323145}
Miguel~A. Arcones, \emph{A {B}ernstein-type inequality for {$U$}-statistics and
  {$U$}-processes}, Statist. Probab. Lett. \textbf{22} (1995), no.~3, 239--247.
  \MR{1323145}

\bibitem{MR1235426}
Miguel~A. Arcones and Evarist Gin\'{e}, \emph{Limit theorems for
  {$U$}-processes}, Ann. Probab. \textbf{21} (1993), no.~3, 1494--1542.
  \MR{1235426}

\bibitem{MR198524}
Leonard~E. Baum and Melvin Katz, \emph{Convergence rates in the law of large
  numbers}, Trans. Amer. Math. Soc. \textbf{120} (1965), 108--123. \MR{198524}

\bibitem{MR1130366}
Tasos~C. Christofides, \emph{Probability inequalities with exponential bounds
  for {$U$}-statistics}, Statist. Probab. Lett. \textbf{12} (1991), no.~3,
  257--261. \MR{1130366}

\bibitem{MR3583992}
J\'{e}r\^{o}me Dedecker and Florence Merlev\`ede, \emph{A deviation bound for
  {$\alpha$}-dependent sequences with applications to intermittent maps},
  Stoch. Dyn. \textbf{17} (2017), no.~1, 1750005, 27. \MR{3583992}

\bibitem{MR3005732}
Xiequan Fan, Ion Grama, and Quansheng Liu, \emph{Large deviation exponential
  inequalities for supermartingales}, Electron. Commun. Probab. \textbf{17}
  (2012), no. 59, 8. \MR{3005732}

\bibitem{MR3311214}
\bysame, \emph{Exponential inequalities for martingales with applications},
  Electron. J. Probab. \textbf{20} (2015), no. 1, 22. \MR{3311214}

\bibitem{MR1857312}
Evarist Gin\'{e}, Rafal Latala, and Joel Zinn, \emph{Exponential and moment
  inequalities for {$U$}-statistics}, High dimensional probability, {II}
  ({S}eattle, {WA}, 1999), Progr. Probab., vol.~47, Birkh\"{a}user Boston,
  Boston, MA, 2000, pp.~13--38. \MR{1857312}

\bibitem{MR1227625}
Evarist Gin\'{e} and Joel Zinn, \emph{Marcinkiewicz type laws of large numbers
  and convergence of moments for {$U$}-statistics}, Probability in {B}anach
  spaces, 8 ({B}runswick, {ME}, 1991), Progr. Probab., vol.~30, Birkh\"{a}user
  Boston, Boston, MA, 1992, pp.~273--291. \MR{1227625}

\bibitem{MR3426520}
Davide Giraudo, \emph{Holderian weak invariance principle under a {H}annan type
  condition}, Stochastic Process. Appl. \textbf{126} (2016), no.~1, 290--311.
  \MR{3426520}

\bibitem{MR3615086}
\bysame, \emph{Holderian weak invariance principle for stationary mixing
  sequences}, J. Theoret. Probab. \textbf{30} (2017), no.~1, 196--211.
  \MR{3615086}

\bibitem{MR0144363}
Wassily Hoeffding, \emph{Probability inequalities for sums of bounded random
  variables}, J. Amer. Statist. Assoc. \textbf{58} (1963), 13--30. \MR{0144363}

\bibitem{MR2073426}
Christian Houdr\'{e} and Patricia Reynaud-Bouret, \emph{Exponential
  inequalities, with constants, for {U}-statistics of order two}, Stochastic
  inequalities and applications, Progr. Probab., vol.~56, Birkh\"{a}user,
  Basel, 2003, pp.~55--69. \MR{2073426}

\bibitem{MR903815}
P.~N. Kokic, \emph{Rates of convergence in the strong law of large numbers for
  degenerate {$U$}-statistics}, Statist. Probab. Lett. \textbf{5} (1987),
  no.~5, 371--374. \MR{903815}

\bibitem{MR1856684}
Emmanuel Lesigne and Dalibor Voln{\'y}, \emph{{Large deviations for
  martingales}}, Stochastic Process. Appl. \textbf{96} (2001), no.~1, 143--159.
  \MR{1856684 (2002k:60080)}

\bibitem{MR614652}
Kuang~Hsien Lin, \emph{Convergence rate and the first exit time for
  {$U$}-statistics}, Bull. Inst. Math. Acad. Sinica \textbf{9} (1981), no.~1,
  129--143. \MR{614652}

\bibitem{MR2336595}
P\'{e}ter Major, \emph{On a multivariate version of {B}ernstein's inequality},
  Electron. J. Probab. \textbf{12} (2007), 966--988. \MR{2336595}

\bibitem{MR2791056}
\bysame, \emph{Estimation of multiple random integrals and {$U$}-statistics},
  Mosc. Math. J. \textbf{10} (2010), no.~4, 747--763, 839. \MR{2791056}

\bibitem{MR740907}
Avi Mandelbaum and Murad~S. Taqqu, \emph{Invariance principle for symmetric
  statistics}, Ann. Statist. \textbf{12} (1984), no.~2, 483--496. \MR{740907}

\bibitem{MR2206313}
Florence Merlev{\`e}de, Magda Peligrad, and Sergey Utev, \emph{{Recent advances
  in invariance principles for stationary sequences}}, Probab. Surv. \textbf{3}
  (2006), 1--36. \MR{2206313 (2007a:60025)}

\bibitem{MR2000642}
Alfredas Ra{\v{c}}kauskas and Charles Suquet, \emph{Necessary and sufficient
  condition for the {L}amperti invariance principle}, Teor. \u Imov\=\i r. Mat.
  Stat. (2003), no.~68, 115--124. \MR{2000642 (2004g:60050)}

\bibitem{MR2054586}
Alfredas Ra\v{c}kauskas and Charles Suquet, \emph{{Necessary and sufficient
  condition for the functional central limit theorem in {H}{\"o}lder spaces}},
  J. Theoret. Probab. \textbf{17} (2004), no.~1, 221--243. \MR{2054586
  (2005b:60085)}

\bibitem{MR606989}
Ludger R\"{u}schendorf, \emph{Ordering of distributions and rearrangement of
  functions}, Ann. Probab. \textbf{9} (1981), no.~2, 276--283. \MR{606989}

\bibitem{MR1736910}
Ch. Suquet, \emph{Tightness in {S}chauder decomposable {B}anach spaces},
  Proceedings of the {S}t. {P}etersburg {M}athematical {S}ociety, {V}ol. {V},
  Amer. Math. Soc. Transl. Ser. 2, vol. 193, Amer. Math. Soc., Providence, RI,
  1999, pp.~201--224. \MR{1736910}

\bibitem{MR1607435}
Henry Teicher, \emph{On the {M}arcinkiewicz-{Z}ygmund strong law for
  {$U$}-statistics}, J. Theoret. Probab. \textbf{11} (1998), no.~1, 279--288.
  \MR{1607435}

\bibitem{MR1655931}
Qiying Wang, \emph{Bernstein type inequalities for degenerate {$U$}-statistics
  with applications}, Chinese Ann. Math. Ser. B \textbf{19} (1998), no.~2,
  157--166, A Chinese summary appears in Chinese Ann. Math. Ser. A {{\bf{1}}9}
  (1998), no. 2, 283. \MR{1655931}

\end{thebibliography}
 \end{document}